\documentclass[twocolumn]{autart}    

\usepackage{graphicx}          
\usepackage{subfigure}

\usepackage{amsmath,amssymb,amsfonts}
\usepackage{algorithmic}
\usepackage{textcomp}
\usepackage{bm}
\usepackage{caption}
\usepackage{comment}
\usepackage{xcolor}
\usepackage{epstopdf}
\usepackage{mathrsfs}
\usepackage[scr=boondox]{mathalfa}
\allowdisplaybreaks[1]
\usepackage{stfloats}
\usepackage{makecell}

\usepackage{enumitem}
\newtheorem{theorem}{Theorem}   
\newtheorem{lemma}{Lemma}
\newtheorem{problem}{Problem}
\newtheorem{proposition}{Proposition}

\newtheorem{example}{Example}
\newtheorem{remark}{Remark}

\newtheorem{assumption}{Assumption}

\DeclareMathOperator{\rank}{rank}

\DeclareMathOperator{\im}{im}

\newcommand{\matH}{\mathcal{H}}
\newcommand{\matT}{\mathcal{T}}

\newcommand{\matAC}{\mathcal{AC}}
\newcommand{\matPC}{\mathcal{PC}}

\begin{document}

\begin{frontmatter}

\title{Online experiment design for continuous-time systems\\ 
	using generalized filtering} 
\vspace*{-3ex}
	
\thanks[footnoteinfo]{This paper was not presented at any IFAC 
meeting. 
This work was partially supported by Jiangsu Provincial Scientific Research Center of Applied Mathematics grant BK20233002. 
H. J. van Waarde acknowledges financial support by the Dutch Research Council under the NWO Talent Programme Veni Agreement (VI.Veni.222.335).
Corresponding author: Simone Baldi. }

\author[BI,cyber]{Jiwei Wang}\ead{jiwei.wang@rug.nl},    
\author[math]{Simone Baldi}\ead{simonebaldi@seu.edu.cn},               
\author[BI]{Henk J. van Waarde}\ead{h.j.van.waarde@rug.nl}  

\address[BI]{Bernoulli Institute for Mathematics, Computer Science and Artificial Intelligence, University of Groningen, The Netherlands}  
\address[cyber]{School of Cyber Science and Engineering, Southeast University, China}             
\address[math]{School of Mathematics, Southeast University, China}        

\vspace*{-2ex}


                                          
\begin{abstract}
The goal of experiment design is to select the inputs of a dynamical system in such a way that the resulting data contain sufficient information for system identification and data-driven control.
This paper investigates the problem of experiment design for continuous-time systems under piecewise constant input signals.
To obviate the need for measuring time derivatives of (data) trajectories, we introduce a generalized filtering framework.
Our main result is to establish conditions on the input and the filter functions under which the filtered data are informative for system identification, i.e., they satisfy a certain rank condition.
We assume that the filter functions are piecewise continuously differentiable, encompassing several filter functions that have appeared in the literature. 
Building on the proposed filtering framework, we adapt an existing experiment design procedure from discrete-time to continuous-time systems, where the piecewise constant input signal is designed online
during system operation.
This method is shown to be sample efficient, in the sense that it deals with the least possible number of filtered data samples for system identification.
Notably, the designed input signal is such that the data capture the system's dynamics at all times between sampling instants, thus establishing a connection with a continuous-time version of \mbox{Willems et al.'s fundamental lemma}.
\vspace*{-3ex}
\end{abstract}

\end{frontmatter}

\section{Introduction}

The growing complexity of modern engineering \mbox{systems} complicates modeling from first principles and motivates the use of data for modeling and control \cite{de2019formulas,van2023informativity,DBLSCT2025}.
Data-driven techniques utilize measurements collected from the system to model the system itself or to design controllers.
Data-driven techniques for discrete-time systems \cite{de2019formulas,van2023informativity,DBLSCT2025,van2020data,berberich2021data,rotulo2022online,wang2025necessary,eising2025data,zhao2025data} have a richer development as compared to their continuous-time counterparts \cite{berberich2021datacontinuous,rapisarda2023orthogonal,ohta2024data,eising2025sampling}, partly due to the challenge of accurately measuring time derivatives of trajectories in continuous-time systems.
One of the possible approaches to address this challenge is to apply one of the various filtering methods proposed in the literature \cite{ohsumi2002subspace,ohta2011stochastic,ohta2024data}.
Alternative methods involve, for instance, robust adaptation \cite{yucelen2011derivative}, estimation using discretization \cite{berberich2021datacontinuous} and orthogonal polynomial bases \cite{rapisarda2023orthogonal}.

While most of the existing literature performs analysis and control using \emph{given} datasets, other \mbox{studies} address a preliminary question---the design of input \mbox{signals}, with the aim of generating suitable data. 
This phase, known as experiment design, is crucial to obtain data containing sufficient information for system identification and data-driven control.
The theoretical \mbox{foundations} for experiment design are provided by the fundamental lemma by Willems et al., originally formulated for discrete-time systems \cite{willems2005note}. 
The lemma characterizes those input signals resulting in collected data that represent all possible trajectories of the system.
Different versions of this lemma for continuous-time systems have been proposed \cite{rapisarda2022persistency,lopez2022continuous,rapisarda2023fundamental,lopez2024input}, demonstrating from different perspectives that a single persistently exciting input signal can generate a system trajectory rich \mbox{enough} to reproduce all other trajectories.
Nevertheless, these results typically rely on direct access to a number of time derivatives of (data) trajectories.
The experiment design for continuous-time \mbox{systems} without the availability of time derivatives remains a challenge.

This naturally raises the question of whether such a problem can be addressed by leveraging on filtering methods like those in \cite{ohsumi2002subspace,ohta2011stochastic,ohta2024data}. 
This work gives a positive answer to this question by means of a generalized filtering framework encompassing existing filtering methods.
Answering this question is far from trivial, as no existing result in the literature guarantees that the data obtained after filtering will retain sufficient information for system identification and data-driven control.
In this work, we provide conditions under which the data obtained after filtering are \emph{informative} for system identification, i.e., they satisfy a certain rank condition that guarantees unique identifiability of the system.

Throughout this paper, we consider two types of data: \emph{sampled data}, obtained from the continuous-time system through piecewise constant input signals and sampling of the trajectories, and \emph{filtered data}, obtained by filtering the trajectories via generalized filter functions.
Based on this, we address two key problems: 
\begin{itemize}
	\item[1)] Find conditions on the trajectories and the filter functions under which the set of filtered data is informative for system identification.
	\item[2)] Construct an experiment design method that generates piecewise constant inputs and guarantees that	the resulting filtered data are informative.
\end{itemize}

In this work, we introduce a generalized filtering framework (see Eq. \eqref{fw}) which encompasses several filtering methods from the literature, see Example 1. 
Based on this framework, the main contribution of the paper is establishing a relation between the filtered data and the sampled data (see Eq. \eqref{relation}).
This relation enables us to state Proposition~\ref{Pranknm}, which provides conditions under which the filtered data are informative for system identification.
Notably, the conditions we present are missing in all filtering methods in the literature, which rely on assuming informativity of the filtered data, rather than deriving the conditions making informativity possible.
Building on the proposed filtering framework, we adapt an online experiment design method from the discrete-time setting in \cite{van2021beyond} to a continuous-time one (see Proposition~\ref{Tis}) and we show that such method guarantees the filtered data to be informative with the minimum possible number of samples (see Theorem \ref{Trank}).
Theorem~\ref{Tpe} shows that the designed input signal is such that the data capture the system's dynamics at all times between sampling instants, which establishes a connection between the proposed method and the continuous-time Willems et al.'s fundamental lemma in \cite{lopez2022continuous}.

The remainder of this paper is organized as follows.
Section 2 introduces the generalized filtering framework.
The problems addressed in the paper are formulated in Section 3.
Section 4 investigates the relation between the filtered data and the sampled data.
The proposed online experiment design method is discussed in Section 5.
Section 6 provides a numerical illustration of the theoretical development.
Finally, Section 7 concludes the paper.

\emph{Notation:}
We denote the set of \emph{non-negative integers} by $ \mathbb{Z}_+ $, the set of \emph{positive integers} by $ \mathbb{N} $, and the set of \emph{non-negative real numbers} by $ \mathbb{R}_{+} $.
Given a matrix \mbox{$ A \in \mathbb{R}^{n\times m} $}, its \textit{Moore-Penrose pseudo-inverse} is denoted by $ A^{\dagger} $.
The \textit{identity matrix} with appropriate dimensions is denoted by $ I $.
Given vectors $ v_k,v_{k+1},\dots,v_\ell\in\mathbb{R}^n $ with $ k \leq \ell $, we define \mbox{$ v_{[k,\ell]} := [v_k\ v_{k+1}\ \cdots\ v_\ell] \in\mathbb{R}^{n\times (\ell-k+1)} $}.

\section{Generalized filtering and filtered data}\label{SFF}

For a given time duration $\mathcal{T} > 0$, we introduce $ M $ filter functions \mbox{$ g_\ell\!: [0,\mathcal{T})\!\to\!\mathbb{R} $} for $ \ell \in \{1,2,\dots,M\} $.
We let $0 = t_0 < t_1 < t_2 < \cdots < t_q = \mathcal{T}$ and assume that $g_\ell$ is continuously differentiable on $[t_{j-1},t_{j})$ for $j = 1,2,\dots,q$. Thus, $g_\ell$ is a piecewise continuously differentiable function. In addition, we assume that $g_\ell(t_j^-) := \lim_{t\uparrow t_j}g_\ell(t)$ exists for all $j = 1,2,\dots,q$.
The functions $ g_\ell $ are used to produce filtered data $ w^{\rm f}_\ell $ from a given signal $ w:[0,\mathcal{T})\to\mathbb{R}^n $.
More precisely, we define
\begin{equation}\label{fw}
w^{\rm f}_\ell:= \int_{0}^{\matT}g_\ell(t)w(t)dt,\quad \ell\in\{1,2,\dots,M\}.
\end{equation}
\begin{remark}
	Note that \eqref{fw} encompasses several filtering methods from the literature.
	For instance, in \cite{ohta2024data}, the following set of filtered data was considered:
	\begin{equation*}
	\mathcal{D}=\{\langle S_\tau h,w\rangle\!:\!h\!\in\!\{h_1,h_2,\dots,h_p\},\tau\!\in\!\{s_1,s_2,\dots,s_r\}\},
	\end{equation*}
	where
	\begin{equation*}
	\langle S_\tau h, w \rangle = \int_\tau^\matT h(t-\tau) w(t) \, dt,
	\end{equation*}
	and $h_j:[0,\mathcal{T})\to\mathbb{R}$ and $s_k \in [0,\mathcal{T})$ for $j = 1,2,\dots,p$ and $k = 1,2,\dots,r$.
	Indeed, \eqref{fw} can accommodate $ \mathcal{D} $ by setting $ M=pr $ and defining
	\begin{equation*}
	g_{(k+(j-1)r)}(t)=
	\left\lbrace 
	\begin{aligned}
	&h_j(t-t_k) && \text{if } t\in [t_k,\matT),\\
	&0 && \text{otherwise,}\\
	\end{aligned}
	\right.
	\end{equation*}
	for $ j\in\{1,2,\dots,p\} $ and $ k\in\{1,2,\dots,r\} $.
\end{remark}

The generality of \eqref{fw} is further elaborated in the next example.
\begin{example}\label{Eff}
	(Filter functions).
	Several filter functions $g_\ell$ considered in the literature can be used in \eqref{fw}.
	For time $ T>0 $ and filter parameter $ \rho>0 $, examples include:
	\begin{itemize}
		\item[i)] compactly supported test functions \cite{ohsumi2002subspace} such as:
		\begin{equation}\label{tf1}
		\!\!g_\ell(t)\!=\!\left\lbrace 
		\begin{aligned}
		&\!\rho(t\!-\!(\ell\!-\!1)T)^2(\ell T\!-\!t)^2 && \!\!\text{if } t\!\in\! [(\ell\!-\!1) T,\ell T),\\
		&\!0 && \!\!\text{otherwise;}\\
		\end{aligned}
		\right.
		\end{equation}
		\item[or]\vspace*{-1ex}
		\begin{equation}\label{tf2}
		g_\ell(t)=\left\lbrace 
		\begin{aligned}
		&e^{-\frac{\rho T^2}{T^2\text{\textminus}(t\text{\textminus}(\ell \text{\textminus}1) T)^2}} && \text{if } t\in [(\ell\!-\!1) T,\ell T),\\
		&0 && \text{otherwise;}\\
		\end{aligned}
		\right.	
		\end{equation}
		\item[ii)] first Laguerre basis function \cite{ohta2024data}:
		\begin{equation}\label{lbf} 
		g_\ell(t)=
		\left\lbrace 
		\begin{aligned}
		&\sqrt{2\rho}e^{\rho((\ell-1)T-t)} && \text{if } t\in [(\ell\!-\!1) T,\matT),\\
		&0 && \text{otherwise;}\\
		\end{aligned}
		\right.
		\end{equation} 
		\item[iii)] low-pass filter function \cite{roy2017combined}:
		\begin{equation}\label{lpf}  
		g_\ell(t)=
		\left\lbrace 
		\begin{aligned}
		&e^{\rho (t-\ell T)} && \text{if } t\in [0,\ell T),\\
		&0 && \text{otherwise.}\\
		\end{aligned}
		\right.
		\end{equation}
	\end{itemize}
\begin{figure}[t]
	\centering
	\subfigure[$ g_\ell(t)=(t-\ell+1)^2(\ell-t)^2 $]
	{\label{ftf1}
		\includegraphics[trim=21bp 261bp 60bp 268bp, clip, height=2.46cm]{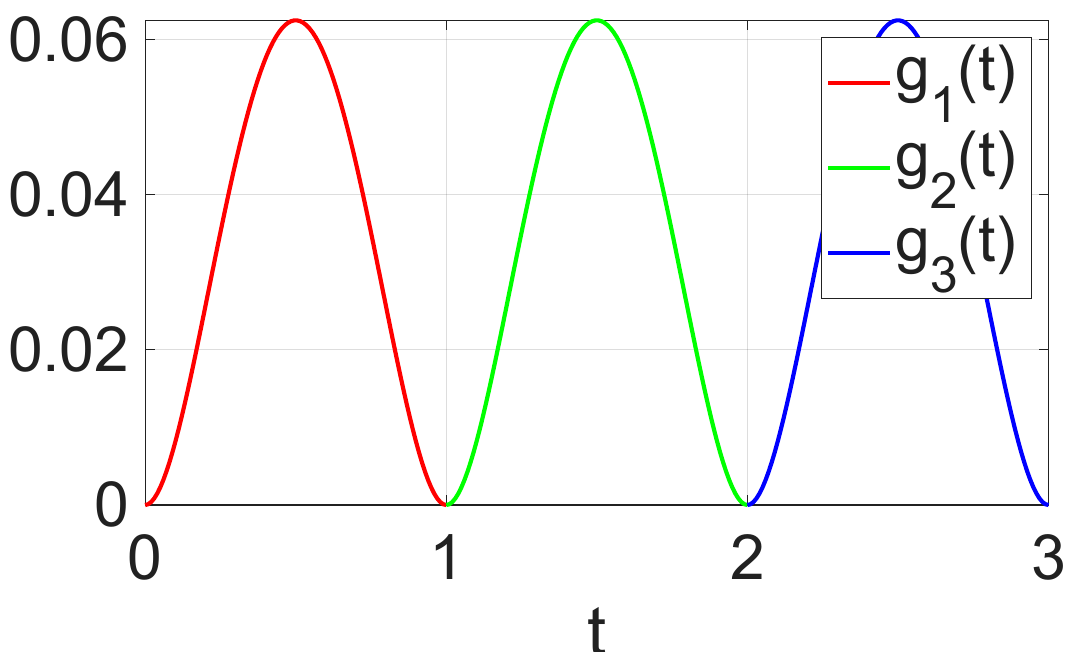}}
	\subfigure[$ g_\ell(t)=e^{\text{\textminus}\frac{1}{1\text{\textminus}(t\text{\textminus}\ell+1 )^2}} $]
	{\label{ftf2}
		\includegraphics[trim=21bp 261bp 60bp 268bp, clip, height=2.46cm]{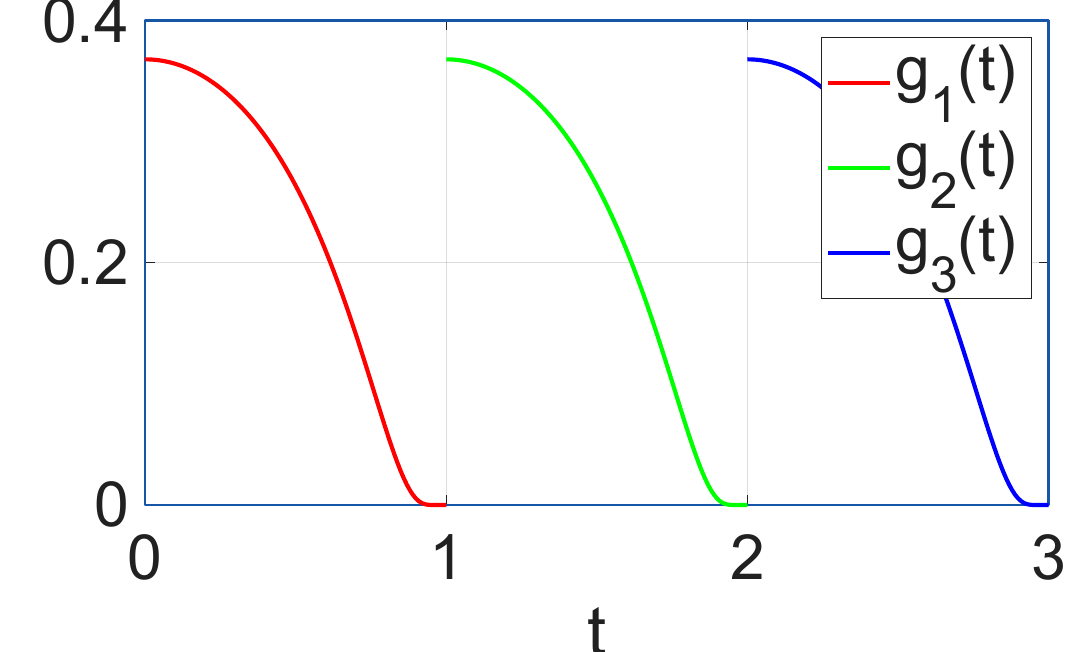}} 
	\subfigure[$ g_\ell(t)=\sqrt{2}e^{\ell\text{\textminus}1\text{\textminus} t} $]
	{\label{flbf}
		\includegraphics[trim=21bp 261bp 60bp 268bp, clip, height=2.46cm]{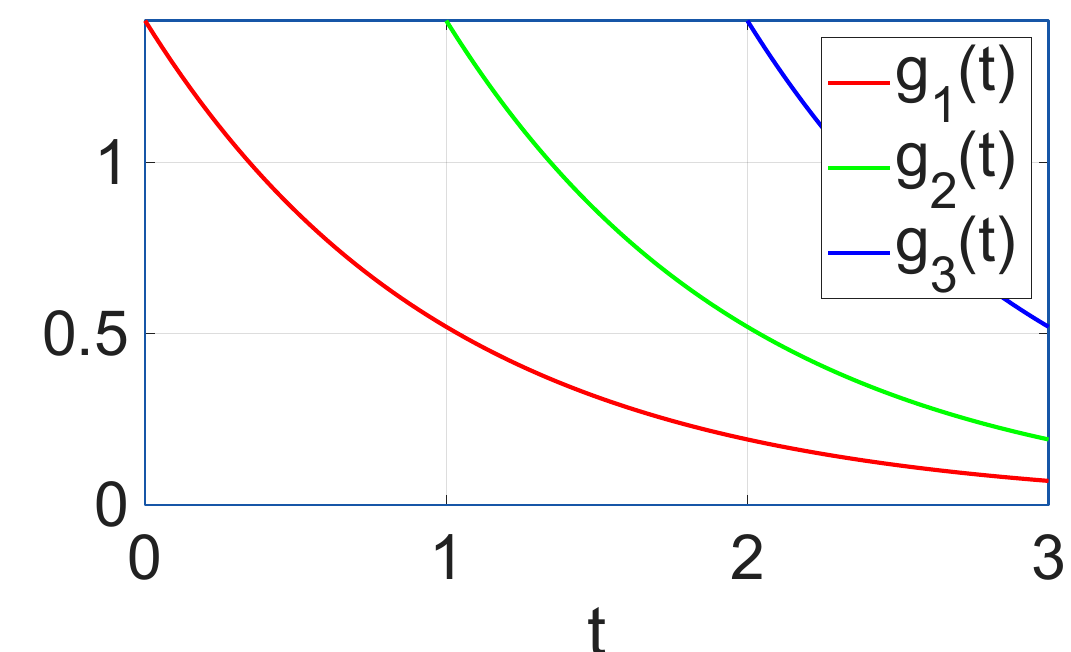}}
	\subfigure[$ g_\ell(t)=e^{t\text{\textminus}\ell} $]
	{\label{flpf}
		\includegraphics[trim=21bp 261bp 60bp 268bp, clip, height=2.46cm]{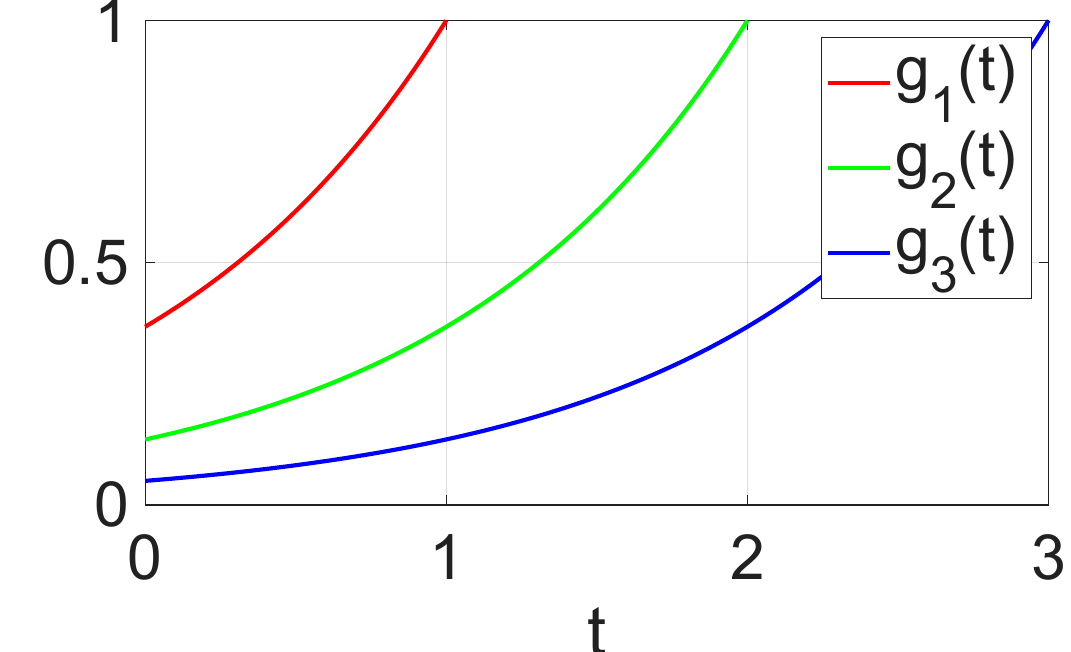}}
	\vspace{-0.2cm}
	\caption{Examples of filter functions for $ \ell \in \{1,2,3\} $:  (a)-(b) compactly supported test functions, (c) first Laguerre basis function, (d) low-pass filter function.}
	\label{fdo}
	\vspace{0.2cm}
\end{figure}
Figures \ref{ftf1}-\ref{flpf} show the four filter functions in \mbox{\eqref{tf1}-\eqref{lpf}} with parameters $ \rho=1 $, $ T=1 $ and $ \ell\in\{1,2,3\} $.
In the low-pass filtering case, the filtered data in \eqref{fw} can be interpreted as the state of a dynamical system sampled at certain time instants.
To illustrate this, consider the dynamical system
\begin{equation}\label{slpf}
\dot{w}^{\rm f}(t)=-\rho{w}^{\rm f}(t)+w(t),\quad w^{\rm f}(0)=0,
\end{equation}
where $ w^{\rm f}(t)\in\mathbb{R}^n $ is the state of system \eqref{slpf}.
The solution of \eqref{slpf} is
\begin{equation}\label{few}
w^{\rm f}(t)=\int_{0}^{t} e^{-\rho (t-\tau)}w(\tau)d\tau.
\end{equation}
Hence, the following relation holds:
\begin{equation}\label{lpfw}
w^{\rm f}(\ell T)=\int_{0}^{\ell T}\! e^{\rho (\tau-\ell T)}w(\tau)d\tau =\int_{0}^{\matT}\! g_\ell(\tau)w(\tau)d\tau=w_\ell^{\rm f},
\end{equation}
where we have used \eqref{fw}, and the specific low-pass filtering function $g_\ell$ in \eqref{lpf}. 
We conclude that the filtered data $ w_\ell^{\rm f} $ obtained from the low-pass filter function \eqref{lpf} are samples of the state of the dynamical system \eqref{slpf} at specific time instants.
The low-pass filter function \eqref{lpf} is often employed to avoid the need for measuring the time derivative of (state) variables \cite{roy2017combined,cho2017composite}. 
Indeed, denote
\begin{equation*}\label{dw}
w^{\rm df}_\ell= \int_{0}^{\matT}g_\ell(\tau)\dot{w}(\tau)d\tau\in \mathbb{R}^n.
\end{equation*}
Then, using integration by parts, we obtain
\begin{equation}\label{calculatedx}
\begin{aligned}
w^{\rm df}_\ell
&=\int_{0}^{\ell T} e^{\rho (\tau-\ell T)}\dot{w}(\tau)d\tau\\
&=w(\ell T)-e^{-\rho\ell T}w(0)-\rho w^{\rm f}_\ell,
\end{aligned}
\end{equation}
i.e., the filtered time derivative data $ w^{\rm df}_\ell $ can be calculated without measuring the time derivative $\dot{w}$. 
\end{example}

\section{Problem formulation}\label{SP}

The filtering approach in \eqref{fw} is now applied to signals generated by a continuous-time system.
Consider the continuous-time linear time-invariant system
\begin{equation}\label{cs}
	\dot{x}(t) = Ax(t) + Bu(t), \quad	x(0) = x_0,
\end{equation}
where $ x(t) \in \mathbb{R}^n $ is the system state, $ u(t) \in \mathbb{R}^m $ is the control input, and $ A \in \mathbb{R}^{n\times n}, B \in \mathbb{R}^{n\times m} $ are  the system matrices, assumed to be unknown.
Throughout the paper, we assume that the pair $ (A, B) $ is controllable.

Let $ \matAC([0,\mathcal{T}), \mathbb{R}^n) $ be the space of absolutely continuous functions from $ [0,\mathcal{T}) $ to $ \mathbb{R}^n $, and let $ \matPC([0,\mathcal{T}), \mathbb{R}^m) $ be the \mbox{space} of piecewise continuous functions from $ [0,\mathcal{T}) $ to $ \mathbb{R}^m $.
We define the behavior \cite{willems1997introduction} of \eqref{cs} as 
\begin{align*}
\mathfrak{B}_\mathcal{T} := \{ (x,u) \in \matAC([0,\mathcal{T}),\mathbb{R}^n) \times \matPC([0,\mathcal{T}),\mathbb{R}^m) \mid \ \\
\eqref{cs} \text{  holds for almost all } t\in [0,\mathcal{T})\}.
\end{align*}
We call elements of $ \mathfrak{B}_\mathcal{T} $ (input-state) trajectories of \eqref{cs} on $ [0,\mathcal{T}) $.

Given a trajectory $(u,x) \in \mathfrak{B}_\mathcal{T}$, we denote the collection of \emph{filtered data} by
\begin{equation}\label{fisd}
\begin{aligned}
x^{\rm df}_{[1,M]}:=&\begin{bmatrix}
x^{\rm df}_1& x^{\rm df}_2& \cdots& x^{\rm df}_M
\end{bmatrix}
\in\mathbb{R}^{n\times M},\\
x^{\rm f}_{[1,M]}:=&\begin{bmatrix}
x^{\rm f}_1& x^{\rm f}_2& \cdots& x^{\rm f}_M
\end{bmatrix}
\in\mathbb{R}^{n\times M},\\
u^{\rm f}_{[1,M]}:=&\begin{bmatrix}
u^{\rm f}_1& u^{\rm f}_2& \cdots& u^{\rm f}_M
\end{bmatrix}\in\mathbb{R}^{m\times M},
\end{aligned}
\end{equation}
where
\begin{align}\label{fdx}
x^{\rm df}_\ell&= \int_{0}^{\matT}g_\ell(\tau)\dot{x}(\tau)d\tau\in \mathbb{R}^n, \\
\label{fx}
x^{\rm f}_\ell&= \int_{0}^{\matT}g_\ell(\tau)x(\tau)d\tau\in \mathbb{R}^n, \\
\label{fu}
u^{\rm f}_\ell&= \int_{0}^{\matT}g_\ell(\tau)u(\tau)d\tau\in \mathbb{R}^m,
\end{align} 
for $ \ell\in\{1,2,\dots,M\} $.
Then, \eqref{fdx} can be further rewritten using integration by parts:
\begin{equation}\label{xdf}
\begin{aligned}
&x^{\rm df}_\ell = \sum_{j=1}^{q} \int_{t_{j-1}}^{t_j^-}g_\ell(\tau)\dot{x}(\tau)d\tau = \\
&\! \sum_{j=1}^{q}\! \Bigg(g_\ell(t_j^-)x(t_j^-) \!-\! g_\ell(t_{j-1})x(t_{j-1})
\!-\! \int_{t_{j-1}}^{t_j^-}\! \dot{g}_\ell(\tau)x(\tau)d\tau\Bigg),
\end{aligned}
\end{equation}
where $ x(t_j^-) $ is well defined for all $ j\in\{1,2,\dots,q\} $ because $ x $ is absolutely continuous.
This shows that there is no need for measuring the time derivative $\dot{x}$ of the \mbox{state} in order to obtain the filtered time derivative $ x^{\rm df}_\ell $.

Since integration is a linear transformation, \eqref{fdx}-\eqref{fu} and \eqref{cs} yield the algebraic equation
\begin{equation}\label{fcs}
x^{\rm df}_{[1,M]}=
Ax^{\rm f}_{[1,M]}+
Bu^{\rm f}_{[1,M]}.
\end{equation}
In this work, we will consider the rank condition
\begin{equation}\label{rank}
	\rank\left(\begin{bmatrix}
	x^{\rm f}_{[1,M]}\\
	u^{\rm f}_{[1,M]}
	\end{bmatrix}\right)
	=n+m.
\end{equation}
The rank condition \eqref{rank} is crucial for system identification and data-driven control based on filtered data, see e.g., \cite{ohta2024data}. 
Indeed, as an example we note that \eqref{fcs} and \eqref{rank} imply that the system matrices $A$ and $B$ can be uniquely recovered from the data as
\begin{equation}\label{identify}
\begin{bmatrix}
A&B
\end{bmatrix}=
x^{\rm df}_{[1,M]}
\begin{bmatrix}
x^{\rm f}_{[1,M]}\\
u^{\rm f}_{[1,M]}
\end{bmatrix}^\dagger,
\end{equation}
i.e., the filtered data $ (x^{\rm df}_{[1,M]}, x^{\rm f}_{[1,M]}, u^{\rm f}_{[1,M]}) $ are informative for system identification.
Analogous versions of the rank condition in \eqref{rank} have also appeared for other types of data in \cite{de2019formulas,berberich2021datacontinuous,rapisarda2023orthogonal}, for example, for samples of continuous-time trajectories \cite{de2019formulas,berberich2021datacontinuous} and coefficient matrices corresponding to certain basis functions \cite{rapisarda2023orthogonal}.

In this paper, our goal will be to establish conditions on $ u $ such that the filtered data \eqref{fisd} resulting from the trajectory \mbox{$ (u,x) \in \mathfrak{B}_\mathcal{T} $} satisfy the rank condition \eqref{rank}.
This leads to the formulation of the following two problems.

\begin{problem}\label{Prank}
	Consider the trajectory $(u,x) \in \mathfrak{B}_\mathcal{T}$ and the filter functions $g_\ell$ for $\ell = 1,2,\dots,M$. Provide conditions on this trajectory and filter functions such that \eqref{rank} holds.
\end{problem}
\begin{problem} \label{Pdesign}
	Consider the filter functions $g_\ell$ for $\ell = 1,2,\dots,M$.
	Design the input signal $ u:[0,\mathcal{T})\to\mathbb{R}^m $ such that for every initial state $x(0) = x_0 \in \mathbb{R}^n$, the resulting trajectory $(u,x) \in \mathfrak{B}_\mathcal{T}$ is such that the filtered data $ (u^{\rm f}_{[1,M]},x^{\rm f}_{[1,M]}) $ satisfy \eqref{rank}.
\end{problem}

Note that for \eqref{rank} to hold, it is necessary that $M \geq n+m$. 
In this paper, we are especially interested in the case that $M = n+m$ because it corresponds to the minimum number of filtered data samples required to achieve \eqref{rank}. 
It will follow from the results of this paper that there indeed exist inputs achieving \eqref{rank} for $M = n+m$.

In the next two sections, we solve Problem~\ref{Prank} and Problem~\ref{Pdesign}, respectively.

\section{Rank condition for filtered data}\label{Smain}

We now let $\mathcal{T} := NT$, $N \in \mathbb{N}$ being the number of samples and $T>0$ the sampling time.
We consider a piecewise constant input trajectory \mbox{$u: [0,\mathcal{T}) \to \mathbb{R}^m$}, defined by
\begin{equation}\label{mu}
u(t+kT)=\mu_k,
\end{equation}
where $t \in [0,T)$ and $\mu_k \in \mathbb{R}^m$ for $k = 0,1,\dots,N-1$.
For the trajectory $ (u,x)\in\mathfrak{B}_{NT} $, we consider \emph{sampled data} at the time instants $t+kT$, captured in the matrices
\begin{equation}\label{ssd}
\chi_{[0,N-1]}(t):=\begin{bmatrix}
\chi_0(t)& \chi_1(t)& \cdots& \chi_{N-1}(t)
\end{bmatrix}
\in\mathbb{R}^{n\times N},
\end{equation}
where $ \chi_k(t)=x(t+kT) $.
We use the shorthand notation $\chi_{[0,N-1]} := \chi_{[0,N-1]}(0)$ and we denote
\begin{equation}\label{isd}
\mu_{[0,N-1]}:=\begin{bmatrix}
\mu_0& \mu_1& \cdots& \mu_{N-1}
\end{bmatrix}\in\mathbb{R}^{m\times N}.
\end{equation}


\subsection{Continuous-time fundamental lemma}\label{Spiecewise}

For a piecewise constant input, exact discretization of the continuous-time system \eqref{cs} with sampling time $T$ results in the discrete-time dynamics \cite{chen2012optimal}
\begin{equation}\label{ds}
\chi_{k+1}=A_{T}\chi_k+B_{T}\mu_k,
\end{equation}
where
\begin{equation}\label{FG_T}
A_{T}:=e^{AT}, \quad B_{T}:=\int_{0}^{T}e^{At}Bdt.
\end{equation}
To ensure that the discrete-time system \eqref{ds} preserves the controllability property of the original continuous-time system \eqref{cs}, the following assumption from the literature is considered.
\begin{assumption} \label{Apathological}
	Let $ \lambda_1, \lambda_2, \dots, \lambda_n $ be the eigenvalues of $ A $.
	Then, for any distinct $ j,l \in \{1,2,\dots,n\} $ and $ q\in\mathbb{N} $, the sampling time $ T $ is such that
	\begin{equation}\label{pathological}
	\lambda_{j} - \lambda_{l} \neq \frac{2q\pi}{T}i,
	\end{equation}
	where $ i $ is the imaginary unit.
\end{assumption}

As highlighted in \cite{chen2012optimal}, Assumption 1 is mild in the sense that \eqref{pathological} holds for all $T$ except for a set of measure zero.
Based on Assumption~\ref{Apathological}, we recall a lemma asserting that controllability is preserved under discretization. 
\begin{lemma} \label{LCO}
	(\cite[Theorem 3.2.1]{chen2012optimal}).
	Let Assumption~\ref{Apathological} hold.
	Then $(A_T,B_T)$ is controllable if $(A,B)$ is controllable.
\end{lemma}
Using this property, we recall the continuous-time fundamental lemma from \cite{lopez2022continuous}.
\begin{lemma} \label{Lrank}
	(\cite[Lemma 1]{lopez2022continuous}).
	Let Assumption~\ref{Apathological} hold.
	Consider the trajectory $ (u,x)\in\mathfrak{B}_{NT} $ with $u$ a piecewise constant input signal in \eqref{mu} satisfying $ \rank(\matH_{n+1}(\mu_{[0,N-1]}))=(n+1)m $, where
	\begin{equation*}
	\matH_{n+1}(\mu_{[0,N-1]})=
	\begin{bmatrix}
	\mu_0& \mu_1& \cdots& \mu_{N-n-1}\\
	\vdots&\vdots&&\vdots\\
	\mu_n& \mu_{n+1}& \cdots& \mu_{N-1}\\
	\end{bmatrix},
	\end{equation*}
	is the Hankel matrix of $\mu_{[0,N-1]}$ of depth $n+1$.
	Then,
	\begin{equation}\label{rankisc}
	\rank\left(\begin{bmatrix}
	\chi_{[0,N-1]}(t)\\
	\mu_{[0,N-1]}
	\end{bmatrix}\right)
	=n+m, \quad \forall t \in [0, T).
	\end{equation}
\end{lemma}
Lemma~\ref{Lrank} provides an experiment design method to achieve \eqref{rankisc} for the sampled data \eqref{ssd}-\eqref{isd}. 
Yet, we stress on the fact that nothing can be concluded about the rank condition \eqref{rank} for the \emph{filtered data}, as required to solve Problems~\ref{Prank} and \ref{Pdesign}.
Hence, we proceed by analyzing the relation between the sampled and the filtered data under piecewise	constant input signals.


\subsection{Rank relation between sampled and filtered data}\label{Sfilter}

We aim to establish a rank relation between the sampled and the filtered data.
To this end, we first analyze the filtered state and filtered input data.
By applying the piecewise constant input in \eqref{mu}, it follows from \eqref{fx} and \eqref{fu} that for any $ \ell\in\{1,2,\dots,M\} $,
\begin{align*}
x_\ell^{\rm f}=
& \sum_{j=0}^{N-1} \int_{(j-1)T}^{j T}g_\ell(\tau)x(\tau)d\tau\\
=& \sum_{j=0}^{N-1}\int_{0}^{T}g_\ell(\tau+j T)x(\tau+j T)d\tau\\
=& \sum_{j=0}^{N-1}\int_{0}^{T}
g_\ell(\tau+j T)e^{A\tau}d\tau\chi_j+\\
&\sum_{j=0}^{N-1}\int_{0}^{T}
g_\ell(\tau+j T)\int_{0}^{\tau}e^{A(\tau-s)}dsd\tau B\mu_j,\\
u_\ell^{\rm f}=
& \sum_{j=0}^{N-1} \int_{(j-1)T}^{j T}g_\ell(\tau)u(\tau)d\tau= \sum_{j=0}^{N-1}\int_{0}^{T}g_\ell(\tau+j T)d\tau \mu_j.
\end{align*}
Now, to build the relation between $ x_\ell^{\rm f} $ and $ \chi_{[0,N-1]} $, we make the following assumption.
\begin{assumption}\label{A}
	For all $ \tau\!\in\! [0,T) $ and \mbox{$ j\!\in\!\{0,1,\dots,N\!-\!1\} $}, the filter function $ g_\ell $ satisfies
	\begin{equation}\label{gf}
	g_\ell(\tau+jT)=g(\tau)f_\ell(jT),
	\end{equation} 
	where $ g:[0,T)\to[0,\infty) $ is a continuously differentiable function such that $ \int_{0}^{T}g(\tau)d\tau>0 $, and $ f_\ell: \mathbb{R}\to\mathbb{R} $ for $ \ell\in\{1,\dots,M\} $.
\end{assumption}
\begin{example}\label{Egf}
	(Validity of Assumption~\ref{A}).
	Suppose that $N \geq M$.
	A decomposition as in \eqref{gf} can be found for all the filter functions in Example~\ref{Eff}, namely,
	\begin{itemize}
		\item[i)] compactly supported test functions:
		\begin{equation}	\label{designtest}
		\begin{aligned}
		&g(t)=\rho t^2(T-t)^2 \text{ or } g(t)=e^{-\frac{bT^2}{T^2 \text{\textminus} t^2}},\\
		&f_\ell(t)=
		\left\lbrace 
		\begin{aligned}
		&1 && \text{if } t\in [(\ell-1) T,\ell T),\\
		&0 && \text{otherwise};\\
		\end{aligned}
		\right.
		\end{aligned}
		\end{equation}
		\item[ii)] first Laguerre basis function:
		\begin{equation}\label{designlb}
		\begin{aligned}
		&g(t)=\sqrt{2\rho}e^{-\rho t},\\
		&f_\ell(t)=
		\left\lbrace 
		\begin{aligned}
		&e^{\rho((\ell-1)T-t)} && \text{if } t\in [(\ell-1) T,\matT),\\
		&0 && \text{otherwise};\\
		\end{aligned}
		\right.
		\end{aligned}
		\end{equation}
		\item[iii)] low-pass filter function:
		\begin{equation}\label{designexp}
		\begin{aligned}
		&g(t)=e^{\rho t},\\
		&f_\ell(t)=
		\left\lbrace 
		\begin{aligned}
		&e^{\rho(t-\ell T)} && \text{if } t\in [0,\ell T),\\
		&0 && \text{otherwise}.\\
		\end{aligned}
		\right.
		\end{aligned}
		\end{equation}
	\end{itemize}
	For all cases, $ g(t)\!\geq\!0 $ for any $ t\!\in\![0,T) $ and $ \int_{0}^{T}\!g(\tau)d\tau\!>\!0 $.
\end{example}

Under Assumption~\ref{A}, we have 
\begin{align}\label{xf}
& x_\ell^{\rm f}= \sum_{j=0}^{N-1}\bar{A}\chi_jf_\ell(jT)+\sum_{j=0}^{N-1}\bar{B}\mu_jf_\ell(jT),\\
\label{uf}
& u_\ell^{\rm f}=
\sum_{j=0}^{N-1}\bar{G} \mu_jf_\ell(jT),
\end{align}
where
\begin{align*}
\bar{A}&=\int_{0}^{T}g(\tau)e^{A\tau}d\tau\in\mathbb{R}^{n\times n},\\ 
\bar{B}&=\int_{0}^{T} g(\tau)\int_{0}^{\tau}e^{A(\tau-s)}Bdsd\tau \in\mathbb{R}^{n\times m},\\
\bar{G}&=I \cdot \int_{0}^{T}g(\tau)d\tau \in\mathbb{R}^{m\times m}.\quad
\end{align*}
This results in the relation
\begin{equation}\label{relation}
\begin{bmatrix} x_{[1,M]}^{\rm f} \\ u_{[1,M]}^{\rm f} \end{bmatrix} = \bar{C} \bar{D} \bar{F},
\end{equation}
where
\begin{align*}
	\bar{C} =& 
	\begin{bmatrix}
	\bar{A}&\bar{B}\\
	0&\bar{G}
	\end{bmatrix}, \quad
	\bar{D} =
	\begin{bmatrix}
	\chi_{[0,N-1]}\\
	\mu_{[0,N-1]}
	\end{bmatrix},\\
	\bar{F} =&
	\begin{bmatrix}
	f_1(0)&f_2(0)&\cdots&f_M(0)\\
	f_1(T)&f_2(T)&\cdots&f_M(T)\\
	\vdots&\vdots&\ddots&\vdots\\
	f_1((N-1)T)&f_2((N-1)T)&\cdots&f_M((N-1)T)\\
	\end{bmatrix}\!.
\end{align*}
Now, to understand under which conditions the rank condition \eqref{rank} holds, we first present two lemmas.
\begin{lemma}\label{LC}
	We have that 
	\begin{equation}\label{lc}
	\rank(\bar{C})=n+m.
	\end{equation}
\end{lemma}
\begin{proof}
	We first show that $ \bar{A} $ is nonsingular.
	We consider the Jordan normal form of $A$, i.e., $ A=P\Lambda P^{-1} $, where $P \in \mathbb{R}^{n \times n}$ is nonsingular,
	\begin{equation*}
	\Lambda=\!
	\begin{bmatrix}
	\Lambda_1&0&\cdots&0\\
	0&\Lambda_2&\cdots&0\\ 
	\vdots&\vdots&\ddots&\vdots\\
	0&0&\cdots&\Lambda_p
	\end{bmatrix}\!,\
	\Lambda_l=\!
	\begin{bmatrix}
	\lambda_l&1&\cdots&0&0\\
	0&\lambda_l&\cdots&0&0\\
	\vdots&\vdots&\ddots&\vdots&\vdots\\
	0&0&\cdots&\lambda_l&1\\
	0&0&\cdots&0&\lambda_l\\
	\end{bmatrix}\! \in\mathbb{R}^{n_l\times n_l},
	\end{equation*}
	$ p $ is the number of Jordan blocks,
	$ l\in\{1,2,\dots,p\} $ and $ \sum_{l=1}^{p}n_{l}=n $.
	Then,
	\begin{equation*}
	e^{At}=\sum_{j=0}^{\infty}\frac{A^j}{j!}t^j=P\sum_{j=0}^{\infty}\frac{\Lambda^j}{j!}t^j P^{-1}=Pe^{\Lambda t}P^{-1},
	\end{equation*}
	implying that
	\begin{equation}\label{singular}
	\begin{aligned}
	&\det\left(\int_{0}^{T} g(\tau) e^{A\tau}d\tau\right)\\
	=&\det(P)\det\left(\int_{0}^{T} g(\tau) e^{\Lambda\tau}d\tau\right) \det(P^{-1})\\
	=&\prod_{{l}=1}^{q}\det\left(\int_{0}^{T} g(\tau) e^{\Lambda_{l}\tau}d\tau\right).
	\end{aligned}
	\end{equation}
	Note that for any $ {l}\in\{1,2,\dots,p\} $,
	\begin{equation*}
	g(\tau)e^{\Lambda_{l}\tau}=\begin{bmatrix}
	g(\tau)e^{\lambda_{l}\tau}&*&\cdots&*\\
	0&g(\tau)e^{\lambda_{l}\tau}&\cdots&*\\
	\vdots&\vdots&\ddots&\vdots\\
	0&0&\cdots&g(\tau)e^{\lambda_{l}\tau}\\
	\end{bmatrix},
	\end{equation*}
	where $\ast$ denotes an entry that is left unspecified.
	By Assumption~\ref{A}, there exist $ t_1,t_2\in(0,T) $ such that $t_1 < t_2$ and $ g(t)>0 $ for any $ t\in[t_1,t_2] $.
	Then, since $ g(\tau)e^{\lambda_{l}\tau}\geq0 $ for any $ \tau\in[0,T) $, we have $ \int_{0}^{T} g(\tau)e^{\lambda_l\tau}d\tau\geq \int_{t_1}^{t_2} g(\tau)e^{\lambda_l\tau}d\tau>0 $, implying $ \det(\int_{0}^{T} g(\tau)e^{\Lambda_l\tau}d\tau)\neq0 $.
	Therefore, by \eqref{singular}, we have that $ \bar{A} $ is nonsingular.
	By Assumption~\ref{A}, $ \bar{G} $ is nonsingular, implying that $ \bar{C} $ is nonsingular, i.e., \eqref{lc} holds.
\end{proof}

\begin{lemma}\label{Lrankpq}
	(\cite[Fact 2.10.14]{bernstein2009matrix})
	For any $ P\in \mathbb{R}^{n_1\times n_2} $ and $ Q\in\mathbb{R}^{n_2\times n_3} $, 
	\begin{equation}\label{crank}
	\rank(PQ)= \rank(Q)-\dim(\ker(P)\cap\im(Q)).
	\end{equation}
\end{lemma}

Now we investigate under which conditions \eqref{rank} holds. 
By Lemma~\ref{LC}, $ \bar{C} $ is nonsingular.
Therefore, \eqref{relation} implies
\begin{equation*}
\rank\left(\begin{bmatrix}
x^{\rm f}_{[1,M]}\\
u^{\rm f}_{[1,M]}
\end{bmatrix}\right)=\rank(\bar{D}\bar{F}),
\end{equation*}
that is, \eqref{rank} holds if and only if $ \rank(\bar{D}\bar{F})=n+m $.
By Lemma~\ref{Lrankpq}, we have
\begin{equation*}
\rank(\bar{D}\bar{F})=\rank(\bar{F})-\dim(\ker(\bar{D})\cap\im(\bar{F})).
\end{equation*}
To arrive at the condition $ \rank(\bar{D}\bar{F})=n+m $, we therefore require that
$$
\rank(\bar{F}) - \dim(\ker(\bar{D})\cap\im(\bar{F})) = n+m.
$$
The latter condition is, in general, difficult to impose by choosing the inputs $\mu_0,\mu_1,....,\mu_{N-1}$. 
The reason is that the states of the system cannot be chosen freely, but depend on the choice of inputs and the (unknown) dynamics \eqref{cs}. 
Therefore, it is in general not possible to design an experiment so that the kernel of $ \bar{D} $ coincides with a given subspace. 
We do note, however, that the problem simplifies in the case that $ N = n+m $. 
In this case, if $ \rank (\bar{D}) = \rank (\bar{F}) = n+m $, then \mbox{$ \ker (\bar{D}) = \{0\} $} and thus $ \rank(\bar{D}\bar{F}) = n+m $. 
In what follows, we summarize our progress in a proposition, which provides a solution to Problem \ref{Prank}.

\begin{proposition}\label{Pranknm}
	Let Assumption~\ref{A} hold.
	Suppose that $N = n+m$ and $ \rank(\bar{F})=n+m $.
	Then, \eqref{rank} holds if and only if $ \rank(\bar{D})=n+m $.
\end{proposition}
\begin{proof}
	By \eqref{relation}, the necessity holds since 
	$$ \rank\left( \begin{bmatrix} x_{[1,M]}^{\rm f} \\ u_{[1,M]}^{\rm f} \end{bmatrix}\right) \leq\rank(\bar{D})\leq n+m. $$
	For sufficiency, $N = n+m$ implies $ \ker(\bar{D})=\{0\} $.
	Then, by Lemma \ref{Lrankpq}, $ \rank(\bar{D}\bar{F})=\rank(\bar{F})=n+m $.
	Since $ \bar{C}\in \mathbb{R}^{(n+m)\times(n+m)} $, Lemma \ref{LC} implies that \eqref{rank} holds.
\end{proof}

\begin{example}\label{Erank}
	(Rank of $ \bar{F} $ for various filter functions).
	Let $ N=M=n+m $.
	For the function $ f_\ell $ in \eqref{designtest},	$ \bar{F}=I. $
	For the function $ f_\ell $ in \eqref{designlb},
	\begin{equation}\label{Flbf}
	\bar{F}=\begin{bmatrix}
	1&0&\cdots&0\\
	e^{-\rho T}&1&\cdots&0\\
	\vdots&\vdots&\ddots&\vdots \\
	e^{-\rho(N-1)T}&e^{-\rho(N-2)T}&\cdots&1
	\end{bmatrix}.
	\end{equation}
	Moreover, for the function $ f_\ell $ in \eqref{designexp},
	\begin{equation}\label{Flpf}
	\bar{F}=\begin{bmatrix}
	e^{-\rho T}&e^{-2\rho T}&\cdots&e^{-\rho NT}\\
	0&e^{-\rho T}&\cdots&e^{-\rho(N-1)T}\\
	\vdots&\vdots&\ddots&\vdots \\
	0&0&\cdots&e^{-\rho T}
	\end{bmatrix}.
	\end{equation}
	Note that $ \rank(\bar{F})=n+m $ in all three cases.
\end{example}
As shown in Example~\ref{Erank}, all principal submatrices of $ \bar{F} $ derived from \eqref{designtest}-\eqref{designexp} are full rank.
This observation, together with \eqref{relation}, leads to a rank relation between the filtered data and the sampled data.
\begin{proposition}\label{Prelation}
	Let Assumption~\ref{A} hold. 
	Suppose that the functions $g_\ell$, for $\ell = 1,2,\dots,M$, satisfy $ f_\ell((\ell-1)T)\neq 0 $ and $ f_\ell(kT)=0 $ for all $k \geq \ell$ and $ \ell\in\{1,2,\dots,\min\{N,M\}\} $.
	Then, 
	\begin{equation}\label{rank=}
	\rank\left(\begin{bmatrix}
	\chi_{[0,k-1]}\\\mu_{[0,k-1]}
	\end{bmatrix}\right)=
	\rank\left(\begin{bmatrix}
	x^{\rm f}_{[1,k]}\\
	u^{\rm f}_{[1,k]}
	\end{bmatrix}\right),
	\end{equation}
	for all $ k\in\{1,2,\dots,\min\{N,M\}\} $.
\end{proposition}
\begin{proof}
	Since $ f_\ell(kT)=0 $ for all $ k\geq\ell $, by considering the first $ k $ columns of \eqref{relation}, we have
	\begin{equation}\label{samerank}
	\begin{bmatrix}
	x^{\rm f}_{[1,k]}\\
	u^{\rm f}_{[1,k]}
	\end{bmatrix}
	=\bar{C}
	\begin{bmatrix}
	\chi_{[0,k-1]}\\\mu_{[0,k-1]}
	\end{bmatrix}\bar{F}_k,
	\end{equation}
	where
	\begin{equation*}
	\bar{F}_k=\begin{bmatrix}
	f_1(0)&f_2(0)&\cdots&f_k(0)\\
	0&f_2(T)&\cdots&f_k(T)\\
	\vdots&\vdots&\ddots&\vdots\\
	0&0&\cdots&f_k((k-1)T)\\
	\end{bmatrix}.
	\end{equation*}
	Since $ f_\ell((\ell-1)T)\neq 0 $ for all $ \ell\in\{1,2,\dots,k\} $, $ \bar{F}_k $ is nonsingular.
	By Lemma~\ref{LC}, $ \bar{C} $ is also nonsingular.
	Hence, \eqref{relation} implies that \eqref{rank=} holds for any \mbox{$ k\in\{1,2,\dots,\min\{N,M\}\} $}.
\end{proof}

	\begin{remark}
		Note that the condition $ f_{\ell}(kT) = 0 $ for all $k \geq \ell$ in Proposition \ref{Prelation} implies that the signal $ x $ (or $ u $) over the interval $ [\ell T,\mathcal{T}) $ does not affect the filtered data $ x^{\rm f}_\ell $ (or $ u^{\rm f}_\ell $).
		By examining the matrix $ \bar{F} $ in Example \ref{Erank} for the filter functions \eqref{designtest}-\eqref{designexp}, it can be observed that the result of Proposition~\ref{Prelation} applies to \eqref{designtest} and \eqref{designexp}, but not to \eqref{designlb}. 
	\end{remark}

In the next subsection, we will focus on an experiment design method that imposes $ \rank(\bar{D})=n+m $ by appropriate choice of the inputs. 
This, in combination with Proposition~\ref{Prelation}, will lead to the desired rank condition on the \emph{filtered} data as in \eqref{rank}.

\section{Online experiment design}
\label{Sdesign}

The basic idea to achieve $ \rank(\bar{D})=n+m $ involves designing the input so that the rank of the data matrix increases progressively, that is, for all \mbox{$ k\in\{1,2,\dots,n+m-1\} $},
\begin{equation}\label{rank+}
\rank\left(\begin{bmatrix}
\chi_{[0,k-1]}&\chi_k\\
\mu_{[0,k-1]}&\mu_k
\end{bmatrix}\right)=
\rank\left(\begin{bmatrix}
\chi_{[0,k-1]}\\
\mu_{[0,k-1]}
\end{bmatrix}\right)+1.
\end{equation}
Condition \eqref{rank+} can be satisfied if the input signal is designed during system operation via the following online procedure inspired by \cite[Theorem 1]{van2021beyond}.
\begin{proposition}\label{Tis}
	Let Assumption~\ref{Apathological} hold.
	For the continuous-time system \eqref{cs}, design the piecewise constant input signal $ u $ in \eqref{mu} as follows. 
	At time $ t=0 $, select a nonzero $ \mu_0 $. 
	At time $ t=kT $, for $ k \in \{1,2,\dots,N-1\} $,
	\begin{itemize}
		\item if
		\begin{equation}\label{isnotin}
		\chi_k \notin \im \chi_{[0,k-1]},
		\end{equation}
		then select $ \mu_k $ arbitrarily;
		\item if 
		\begin{equation}\label{isin}
		\chi_k \in \im \chi_{[0,k-1]},
		\end{equation} 
		then there exist $ \eta \in \mathbb{R}^m $ and $ \xi\in \mathbb{R}^n $ with $ \eta\neq0 $ such that
		\begin{equation}\label{v}
		\begin{bmatrix}
		\xi^\top&\eta^\top
		\end{bmatrix}
		\begin{bmatrix}
		\chi_{[0,k-1]}\\\mu_{[0,k-1]}
		\end{bmatrix}=0.
		\end{equation}
		In this case, select $ \mu_k $ such that
		\begin{equation}\label{muk}
		\xi^\top \chi_k+\eta^\top \mu_k\neq0.
		\end{equation}
	\end{itemize}
	Then, the resulting trajectory $ (u,x)\in\mathfrak{B}_{(n+m)T} $ satisfies $ \rank(\bar{D})=n+m $.
\end{proposition}
Proposition~\ref{Tis} follows from \cite[Theorem 1]{van2021beyond}. 
Specifically, design of the piecewise constant values is equivalent to designing discrete-time inputs for the system \eqref{ds}.
Its rationale is as follows.
For a piecewise constant input signal $ u $, the sampled data obey the dynamics in \eqref{ds}.
Then, by Assumption~\ref{Apathological} and Lemma~\ref{LCO}, at time instant $ kT $, $ k\in\{0,1,\dots,N-1\} $, a $\mu_k$ can always be designed such that \eqref{rank+} holds.
Note that one may select a solution which satisfies a norm constraint, e.g., to account for actuation constraints.
Specifically, let $\epsilon > 0$ be an upper bound for the norm of the input. Then, it can be shown that for $k \in \{0,1,\dots,N-1\}$, there exists a $\mu_k$ satisfying \eqref{muk} and $|| \mu_k || \leq \epsilon$.

We are now in a position to solve Problem~\ref{Pdesign}.
\begin{theorem}\label{Trank}
	Let Assumption~\ref{Apathological} and \ref{A} hold, and suppose that $ \rank(\bar{F})=n+m $. Let the input signal $ u $ be designed according to Proposition~\ref{Tis}.
	Then, the resulting trajectory $(u,x) \in \mathfrak{B}_\mathcal{T}$ is such that the filtered data $ (u^{\rm f}_{[1,M]},x^{\rm f}_{[1,M]}) $ satisfy \eqref{rank}.
\end{theorem}
\begin{proof} 
	Since $ u $ is designed according to Proposition~\ref{Tis}, the resulting trajectory $ (u,x)\in\mathfrak{B}_{(n+m)T} $ satisfies $ \rank(\bar{D})=n+m $.
	Then, we conclude from Proposition~\ref{Pranknm} that \eqref{rank} is satisfied since $ N=n+m $ and $ \rank(\bar{F})=n+m $.
\end{proof}

Let us now go one step further. 
The following result shows that the online input design ensures not only $ \rank(\bar{D})=n+m $, but that the stronger condition in \eqref{rankisc} holds.
\begin{theorem}\label{Tpe}
	Let Assumption~\ref{Apathological} hold.
	Consider the trajectory $ (u,x)\in\mathfrak{B}_{NT} $ where $u$ is a piecewise constant input signal designed as in Proposition~\ref{Tis}.
	Then, \eqref{rankisc} holds.
\end{theorem}
\begin{proof}
	Let $ t\in[0,T) $, and $ \eta\in\mathbb{R}^m $ and $ \xi\in\mathbb{R}^n $ be vectors such that
	\begin{equation}
	\begin{bmatrix}
	\xi^\top&\eta^\top
	\end{bmatrix}
	\begin{bmatrix}
	\chi_{[0,N-1]}(t)\\
	\mu_{[0,N-1]}
	\end{bmatrix}=0.
	\end{equation}
	By \eqref{cs}, for any $ k\in\{1,2,\dots,N-1\} $ we get
	\begin{equation}\label{}
	\begin{aligned}
	\chi_k(t)&=e^{At}\chi_k+\int_{kT}^{t+kT}e^{A(t+kT-\tau)}Bu(\tau)d\tau\\
	&=e^{At}\chi_k+\int_{0}^{t}e^{A(t-\tau)}Bd\tau \mu_k.
	\end{aligned}
	\end{equation}
	Then, we have
	\begin{equation}
	\chi_{[0,N-1]}(t)=
	\begin{bmatrix}
	e^{At}&\int_{0}^{t}e^{A(t-\tau)}Bd\tau 
	\end{bmatrix}
	\begin{bmatrix}
	\chi_{[0,N-1]}\\
	\mu_{[0,N-1]}
	\end{bmatrix},
	\end{equation}
	implying that
	\begin{equation}
	\begin{bmatrix}
	\xi^\top&\eta^\top
	\end{bmatrix}
	\begin{bmatrix}
	e^{At}&\int_{0}^{t}e^{A(t-\tau)}Bd\tau \\
	0&I
	\end{bmatrix}
	\begin{bmatrix}
	\chi_{[0,N-1]}\\
	\mu_{[0,N-1]}
	\end{bmatrix}=0.
	\end{equation}
	According to Proposition~\ref{Tis}, we have $ \rank(\bar{D})=n+m $, implying that
	\begin{equation}
	\xi^\top e^{At}=0\ \text{ and }\ \xi^\top\int_{0}^{t}e^{A(t-\tau)}Bd\tau +\eta^\top=0.
	\end{equation}
	Since $ e^{At} $ is nonsingular, we have $ \xi=0 $, implying $ \eta=0 $.
	Hence, we conclude that \eqref{rankisc} holds.
\end{proof}

Theorem \ref{Tpe} shows that the designed input signal is such that the data capture the system's dynamics at all times between sampling instants, which establishes a connection between the proposed method and the continuous-time Willems et al.'s fundamental lemma in \cite{lopez2022continuous}.
Next, we compare Theorem~\ref{Tpe} to Lemma~\ref{Lrank}, i.e., \cite[Lemma 1]{lopez2022continuous}. 
Lemma~\ref{Lrank} requires at least $ n+m+nm $ samples to guarantee \eqref{rankisc}. 
However, Theorem~\ref{Tpe} guarantees \eqref{rankisc} with the minimum number of samples possible, namely, $n+m$.

\section{Numerical example}

In this section, we consider a numerical example to demonstrate that, using the experiment design method in Section~\ref{Sdesign}, both the sampled data and the filtered data generated by different filters are informative for system identification.
The example is concluded by showing that the system matrices can be reconstructed from the filtered data.

The system to be identified is described by the aircraft longitudinal dynamics in \cite{moustakis2018adaptive} with system matrices
\vspace{-1ex}
\begin{align*}
A=&\begin{bmatrix}
-0.0190 & 0.0825 & -0.1005 & -0.3206 \\
-0.2154 & -2.7859 & 1.2031 & -0.0271 \\
3.2527 & -30.7871 & -3.5418 & 0 \\
0 & 0 & 1 & 0
\end{bmatrix},\\
B=&\begin{bmatrix}
0.0065 & 0.0534 \\
-0.6103 & 0.0020 \\
-74.6355 & 0.5431 \\
0 & 0
\end{bmatrix}.
\end{align*}
The four states of the system are the velocity along the $ x $- and $ z $-body-axis, the angular velocity around the \mbox{$ y $-body-axis}, and the pitch angle, while the two \mbox{inputs} are the elevator deflection and throttle.
Note that \mbox{$n+m=6$}.
Let $ T=0.1 $.
Because the eigenvalues of $A$ are \mbox{$\{  -3.1548 \pm 6.0892i, -0.0185 \pm 0.3263i\}$}, Assumption~\ref{Apathological} holds. 
In general, even though $A$ is unknown, Assumption~\ref{Apathological} fails to hold only on a set of measure zero.
Let $ x(0)=\left[ 2\ -\!1\ 1\ 0.5 \right]^\top $.
Using the design in Proposition~\ref{Tis}, the piecewise constant input is selected as
\begin{equation}\label{input}
	\mu_{[0,5]}=\begin{bmatrix}
		0.5 & -0.2 & 0 & 0 & 0.2 & -0.5 \\
		-1 & 0 & 0.5 & -0.5 & 0 & 1
	\end{bmatrix}\!,
\end{equation}
which generates the following state samples:
\begin{align}\nonumber
&	\chi_{[0,6]}=\\
\label{oxd}
& \!\!\!
\begin{bmatrix}
	2 & 1.9600 & 1.9099 & 1.8420 & 1.7680 & 1.7089 & 1.6472 \\
	-1 & -0.7525 & -0.3722 & 0.0484 & 0.2962 & 0.2863 & 0.3779 \\
	1 & 0.2743 & 3.4680 & 3.3281 & 2.2990 & 0.0192 & 2.8617 \\
	0.5 & 0.5680 & 0.7746 & 1.1250 & 1.4098 & 1.5188 & 1.6739
\end{bmatrix}\!.
\end{align}
We consider two specific examples of filter functions: the first Laguerre basis function in \eqref{lbf} with $\rho=1$ and $T = 0.1$, and the low-pass filter function in \eqref{lpf} with $\rho=1$ and $T = 0.1$.
With the first Laguerre basis function, by \eqref{fx} and \eqref{fu}, we obtain the following filtered data matrices
\begin{align}
\nonumber
&x^{\rm f}_{[1,6]}=\\
\label{fxd1}
&
\begin{bmatrix}
	1.1832 & 1.0132 & 0.8315 & 0.6398 & 0.4385 & 0.2263 \\
	-0.1333 & -0.0187 & 0.0669 & 0.0969 & 0.0794 & 0.0410 \\
	1.2364 & 1.2644 & 1.0940 & 0.6878 & 0.3352 & 0.2055 \\
	0.6444 & 0.6319 & 0.6027 & 0.5251 & 0.3909 & 0.2115
\end{bmatrix}\!,
\end{align}
\begin{align}
\nonumber
&u^{\rm f}_{[1,6]}=\\
&\!\!\!
\begin{bmatrix}
	0.0202 & -0.0521 & -0.0278 & -0.0307 & -0.0340 & -0.0673 \\
	-0.0477 & 0.0960 & 0.1061 & 0.0429 & 0.1218 & 0.1346
\end{bmatrix}\!. \quad
\end{align}
From \eqref{xdf}, we can obtain that for any $ \ell \in \{1,2,\dots,6\} $, $x^{\rm df}_\ell = \sqrt{2}e^{0.1\ell-0.7}\chi_6 - \sqrt{2}\chi_{\ell-1} + x^{\rm f}_\ell$.
Then, by \eqref{oxd} and \eqref{fxd1}, we get the filtered state derivatives
\begin{align}
\nonumber
&x^{\rm df}_{[1,6]}=\\
&\!\!\!
\begin{bmatrix}
	-0.3667 & -0.3457 & -0.3080 & -0.2395 & -0.1545 & -0.0826 \\
	1.5742 & 1.3697 & 0.9515 & 0.4244 & 0.0981 & 0.1198 \\
	2.0432 & 3.3310 & -1.0977 & -1.0208 & 0.3973 & 3.8403 \\
	1.2364 & 1.2644 & 1.0940 & 0.6878 & 0.3352 & 0.2055
\end{bmatrix}\!.
\end{align}
With the low-pass filter function, by \eqref{fx} and \eqref{fu}, we obtain the following data matrices
\begin{align}\nonumber
&x^{\rm f}_{[1,6]}=\\
\label{fxd2}
&\!\!\!
\begin{bmatrix}
	0.1883 & 0.3548 & 0.4995 & 0.6237 & 0.7296 & 0.8200 \\
	-0.0819 & -0.1296 & -0.1312 & -0.1006 & -0.0612 & -0.0262 \\
	0.0641 & 0.2572 & 0.5660 & 0.7824 & 0.8098 & 0.8826 \\
	0.0515 & 0.1081 & 0.1886 & 0.2923 & 0.4057 & 0.5169
\end{bmatrix}\!,\\
\nonumber
&u^{\rm f}_{[1,6]}=\\
\label{fud2}
&\!\!\!
\begin{bmatrix}
	0.0476 & 0.0240 & 0.0217 & 0.0197 & 0.0368 & -0.0143 \\
	-0.0952 & -0.0861 & -0.0303 & -0.0750 & -0.0679 & 0.0337
\end{bmatrix}\!,
\end{align}
From \eqref{xdf}, we can obtain that for any $ \ell \in \{1,2,\dots,6\} $, $x^{\rm df}_\ell = \chi_\ell-e^{-0.1\ell}\chi_0-x^{\rm f}_\ell$.
Then, by \eqref{oxd} and \eqref{fxd2}, we get the filtered state derivatives
 \begin{align}\nonumber
&x^{\rm df}_{[1,6]}=\\
\label{fdxd2}
&\!\!\!
\begin{bmatrix}
	-0.0381 & -0.0824 & -0.1391 & -0.1964 & -0.2338 & -0.2704 \\
	0.2343 & 0.5761 & 0.9205 & 1.0671 & 0.9540 & 0.9529 \\
	-0.6946 & 2.3922 & 2.0213 & 0.8463 & -1.3972 & 1.4303 \\
	0.0641 & 0.2572 & 0.5660 & 0.7824 & 0.8098 & 0.8826
\end{bmatrix}\!.
\end{align}
Note that filtered data matrices \eqref{fxd1}-\eqref{fdxd2} are constructed solely for identification purposes and differ in general from the system trajectories obtained from the system.
For both filter functions, we verify that
\begin{equation}\label{r}
\rank\left(\begin{bmatrix}
x^{\rm f}_{[1,6]}\\
u^{\rm f}_{[1,6]}
\end{bmatrix}\right)=
\rank\left(\begin{bmatrix}
\chi_{[0,5]}\\
\mu_{[0,5]}
\end{bmatrix}\right)=6.
\end{equation}
This is in line with the conclusion of Proposition~\ref{Pranknm}.
In other words, the designed input guarantees the filtered data to be informative with the minimum possible number of samples.
Being the filtered data informative for system identification, we can reconstruct the system matrices as in \eqref{identify}:
\begin{equation*}
\begin{bmatrix}
\hat{A}&\hat{B}
\end{bmatrix}=x^{\rm df}_{[1,6]}
\begin{bmatrix}
x^{\rm f}_{[1,6]}\\
u^{\rm f}_{[1,6]}
\end{bmatrix}^{-1}.
\end{equation*}
The Frobenius norm of the identification error is 
\mbox{$
\lVert[A\ B]-[\hat{A}\ \hat{B}]\rVert_F=1.8592\times 10^{-6}
$}
for the first Laguerre basis function, and \mbox{$ 1.5675\times 10^{-6} $} for the low-pass filter function. 
We mention that similar conclusions can be drawn for the filter functions in \eqref{tf1} and \eqref{tf2}, which are not shown for the sake of brevity.

We conclude the numerical validation by showing a noisy data setting. Specifically, we inject additive noise on the input as $\dot{x}(t)=A x(t) + B (u(t)+w(t))$, with the same $(A,B)$ as before. 
The continuous-time noise $w$ is approximated using an integration step $dt=10^{-8}$s with zero-order-hold and discretized noise generated from the normal distribution
$
\mathcal N
\bigl(
0,\,
dt \, \sigma^2 \, \text{diag}(0.25,1)
\bigr)
$ 
for noise levels:
$
\sigma=10^{-5},10^{-4},10^{-3},10^{-2}.
$
The term diag$(0.25, 1)$ is used to reflect the different magnitude of the two inputs in \eqref{input}.

As unfiltered baseline, we collect estimated data of $\dot{x}$ using a finite difference method with precision $\Delta t=10^{-3}$s, that is, 
$$
\hat{\dot{x}}(t)=\frac{x(t)-x(t-\Delta t)}{\Delta t}, \qquad	\forall t\in\{0.1,0.2,\dots,0.6\}.
$$
Then, to examine how filtering affects the identification error, we use (13)--(15) with low-pass filter (5) with $T = 0.1$ and $ \rho = 1 $.
We use the same input in (50) as we have verified that it satisfies the procedure in Proposition~\ref{Tis} for all tested levels of process noise. 
For one realization of the noise, we use the unfiltered and the filtered data and we then calculate the Frobenius norm of the identification error \(\lVert[A\ B]-[\hat{A}\ \hat{B}]\rVert_F\) for different noise levels, reported in the following table.

\begin{table}[htbp]
	\centering
	\begin{tabular}{|c|c|c|c|c|}
		\hline
		& \(\!\sigma=10^{-5}\!\) & \(\!\sigma=10^{-4}\!\) & \(\!\sigma=10^{-3}\!\) & \(\!\sigma=10^{-2}\!\) \\
		\hline
		\makecell{finite\\difference} 
		& 0.2271 & 0.3666 & 10.2321 & 60.7589 \\
		\hline
		\makecell{low-pass\\filtering} 
		& 0.0016 & 0.0414 & 0.1042 & 3.6051 \\
		\hline
	\end{tabular}
	\label{tab:simple_table}
\end{table}

Not surprizingly, the lower the noise level, the smaller the identification error. 
Most importantly, in all cases, the identification error obtained with low-pass filter function is smaller than the identification error obtained with finite-difference state derivative data.
This demonstrates a beneficial effect of filtering in reducing the detrimental effect of noise.

\section{Conclusion}

This paper has investigated the problem of \mbox{experiment} design for continuous-time systems.
To avoid reliance on time derivatives of measured trajectories, which is a key challenge in data-driven methods for continuous-time systems, we have proposed a generalized filtering framework to collect filtered data.
We have presented conditions to ensure that these filtered data are informative for system identification, and adapted an online experiment design method that guarantees informativity with the minimum number of samples.
Several examples of filter functions have demonstrated the generality of the proposed filtering framework.
Potential future work includes applying the proposed experiment design method in specific data‑driven control settings such as data‑driven stabilization or model reference control, and extending the framework to noisy data settings.

\bibliographystyle{abbrv}       
\bibliography{References}          



\end{document}